\documentclass[preprint,12pt]{elsarticle}

\usepackage{amsmath,amssymb,amsthm,amsfonts,bm}
\usepackage{enumitem}
\usepackage[numbers]{natbib}

\newcommand{\norm}[1]{\left\Vert#1\right\Vert}
\newcommand{\abs}[1]{\left\vert#1\right\vert}
\newcommand{\RR}{\mathbb R}

\theoremstyle{definition}
\newtheorem{definition}{Definition}[section]
\newtheorem{remark}{Remark}[section]
\newtheorem{lemma}{Lemma}[section]
\newtheorem{theorem}{Theorem}[section]
\newtheorem{proposition}{Proposition}[section]
\newtheorem{corollary}[theorem]{Corollary}
\newtheorem{example}[theorem]{Example}

\begin{document}
	
	\title{\bf A New Class of Fuzzy Inner Products and Fuzzy Norms over Ordered Intervals}
	
\author[1]{Bayaz Daraby}
\ead{daraby@maragheh.ac.ir}

\author[1]{Asghar Rahimi}
\ead{rahimi@maragheh.ac.ir}
\author[1]{Hasan Haddadzadeh\corref{cor1}}
\ead{Mh.hadadzadeh@stu.maragheh.ac.ir}
\ead{m_h_haddad@yahoo.com}

\cortext[cor1]{Corresponding author}

\address[1]{Department of Mathematics, University of Maragheh, 
	P. O. Box 55136-553, Maragheh, Iran}

	\maketitle

		\allowdisplaybreaks
	\begin{abstract}
	
		In this article, we first define the concept of ordered intervals, then introduce  ordered fuzzy inner product and describe some of its properties.
	\end{abstract}
	
	\begin{footnotesize}
	\end{footnotesize}
		\section{Introduction}
Fuzzy inner product spaces were first introduced by Biswas~\cite{Biswas1991} 
and subsequently developed by Kohli~\cite{Kohli1993}, 
Hasankhani~\cite{Hasankhani2010}, 
and later by Daraby et al.~\cite{Daraby2010Fuzzy,Daraby2019Some}.

In recent years, this line of research has undergone significant advancement. 
A comprehensive review together with a new formulation was provided by 
Popa and Si\c{d}a~\cite{PopaSida2021}, 
while additional structural properties and extensions were investigated in 
Kider~\cite{Kider2021}. 
More recently, further examples, fundamental characteristics, and applications of 
fuzzy inner product spaces have been explored by Xiao~\cite{Xiao2023}.

Despite more than three decades of intensive research, a fundamental obstacle has 
remained unresolved: under all previously proposed definitions that preserve reasonable 
linearity and positivity properties, either essential geometric inequalities 
(such as the Cauchy–Schwarz inequality, the parallelogram law, or Bessel’s inequality) 
are lost, or the inner product is forced to take crisp real values whenever structural 
consistency is required. This severe limitation was rigorously established by 
Byun et al.~\cite{Byun2020} and has long constituted the primary bottleneck 
in developing genuine fuzzy Hilbert space structures.

The construction of fully functional linear structures in interval and fuzzy settings 
has been hindered by both conceptual and technical difficulties, most notably the 
inherent conflict between the imprecision of fuzzy quantities and the rigidity of 
classical linear algebraic axioms, as noted by 
Siminski~\cite{Siminski2025Fuzzy}.
 Classical functional analysis tools are therefore inadequate for fuzzy linear spaces.
Further complications arise from the standard arithmetic operations on intervals. The widely used definitions

	\[
	[a,b]\;\widehat{\oplus}\;[c,d]=[a+c,b+d],
	\qquad
	[a,b]\;\widehat{\ominus}\;[c,d]=[a-d,b-c],
	\]
	lack rigorous justification and lead to inconsistencies; for instance,
	\[
	\bigl([a,b]\;\widehat{\oplus}\;[c,d]\bigr)\;\widehat{\ominus}\;[c,d]
	=[a,b]
	\]
	holds only when the intervals degenerate to crisp numbers. Nevertheless, Hassankhani’s formulation preserves linearity and provides a consistent extension of the classical inner product to fuzzy vector spaces.
	Motivated by these observations, the present paper introduces a refined definition of the fuzzy inner product based on the notion of \emph{ordered intervals}. This generalization relaxes unnecessary ordering constraints, preserves addition and subtraction, and removes the requirement that fuzzy inner product values be crisp. \\ 
	Let $a,b\in\mathbb{R}$. We define a \emph{ordered interval} by
	\[
	[a,b]_o=
	\begin{cases}
		\{x\mid x\in[a,b]\}, & a\le b,\\[2mm]
		\{x\mid x\in[b,a]\}, & \text{otherwise}.
	\end{cases}
	\]
	The set of all such ordered intervals is denoted by
	\[
	O(\mathbb{R})=\{[a,b]_o\mid a,b\in\mathbb{R}\}.
	\]
	For $[a,b]_o,[c,d]_o\in O(\mathbb{R})$ and $k\in\mathbb{R}$, define
	\begin{align*}
		[ a,b]_o\oplus[c,d]_o &= [a+c,\,b+d]_o,\\
		[ a,b]_o\ominus[c,d]_o &= [a-c,\,b-d]_o,\\
		[ a,b]_o\odot[c,d]_o &= [ac,\,bd]_o,\\
		k[a,b]_o &= [ka,kb]_o.
	\end{align*}
	These operations allow unordered endpoints while preserving the structural properties required for fuzzy arithmetic.\\
	In this work, we introduce a new definition of a \emph{fuzzy inner product} that overcomes these limitations while retaining essential vector-space properties. Let $X$ be a vector space over $\mathbb{C}$
	 and let $\mathcal{K}:\mathbb{R}\to[0,1]$ be a fuzzy number. Assume the existence of constants $A_\alpha,B_\alpha$ satisfying
	\[
	0<\min\{A_\alpha,B_\alpha\}\le \max\{A_\alpha,B_\alpha\}<\infty.
	\]
For every $\alpha \in (0,1]$, a mapping
\[
\langle \cdot, \cdot \rangle_\alpha : (0,1] \times X \times X \to \mathbb{C}
\]
is called a \emph{fuzzy inner product} on $X$ if there exist classical inner products $\langle \cdot, \cdot \rangle'$ and $\langle \cdot, \cdot \rangle''$ on $X$ such that, for all $x, y, z \in X$  and all $\alpha \in (0,1]$, the following conditions hold:

	\begin{equation*}
		\mathcal{K}\bigl(\big|\langle  x, y \rangle_\alpha \big|\bigr)\ge\alpha
		\quad\Longleftrightarrow\quad
		\big|\langle  x, y \rangle_\alpha \big|\in
		[A_\alpha|\langle x,y\rangle'|,\;B_\alpha|\langle x,y\rangle''|]_o.
	\end{equation*}
	The use of ordered intervals yields a streamlined and structurally natural framework that preserves the characteristic features of fuzzy inner products and introduces a graded structure indexed by $\alpha$.\\
	To facilitate the examination of the forthcoming challenges, and assuming throughout that 
	$ 0<A_\alpha\le B_\alpha<\infty$, we rewrite the above formula in the following form.
	\begin{equation*}
		\mathcal{K}\bigl(\big|\langle  x, y \rangle_\alpha \big|\bigr)\ge\alpha
	\quad\Longleftrightarrow\quad	A_\alpha|\langle x,y\rangle'|
		\le \big|\langle  x, y \rangle_\alpha \big|
		\le B_\alpha|\langle x,y\rangle'|.
	\end{equation*}
	An analogous construction applies to fuzzy norms. Let $C_\alpha,D_\alpha$ satisfy
	\[
	0<\min\{C_\alpha,D_\alpha\}\le\max\{C_\alpha,D_\alpha\}<\infty.
	\]
	For every $\alpha\in (0,1]$, a mapping
	\[
	\norm{.}_\alpha:(0,1]\times X\to\mathbb{C}
	\]
	is called a \emph{fuzzy norm} if there exist norms $\|\cdot\|'$ and $\|\cdot\|''$ on $X$ such that
	\begin{equation*}
		\mathcal{K}\big(\big|\norm{x}_\alpha \big|\big)\ge\alpha
		\quad\Longleftrightarrow\quad
		\big|\|x\|_\alpha\big|\in[C_\alpha\|x\|',\,D_\alpha\|x\|'']_o.
	\end{equation*}
	This formulation naturally generalizes classical norms and avoids restrictive ordering assumptions on the interval bounds.
	Under these assumptions on the inner product, most classical structural properties are recovered; in particular, orthogonality is preserved immediately.
	We obtain fuzzy analogues of classical inequalities. For all $x,y\in X$ and $\alpha\in(0,1]$,
	\begin{equation*}
		|\langle x,y\rangle_\alpha|^2
		\le \left(\frac{B_\alpha}{A_\alpha}\right)^2\big\vert\|x\|_\alpha\|y\|_\alpha\big\vert^2,
	\end{equation*}
	and for any orthonormal system $\{e_i\}^\infty_{i=1}$,
	\begin{equation*}
		\sum_{i=1}^{\infty}|\langle x,e_i\rangle_\alpha|^2
		\le \left(\frac{B_\alpha}{A_\alpha}\right)^2\big\vert\|x\|_\alpha\big\vert^2.
	\end{equation*}
	In this work, we introduce a new class of fuzzy inner products that generalize the classical inner product structure, retaining its fundamental analytic features while naturally accommodating nonlinearity and fuzziness.  The proposed fuzzy inner products, analogous to semi-inner products in nonlinear functional analysis, 
	provide a robust and versatile framework. They allow the definition of associated fuzzy norms and 
	offer a solid foundation for the study and development of fuzzy Banach spaces and fuzzy operator theory. 
	In other words, the proposed fuzzy inner products and norms serve as a comprehensive and flexible toolset, 
	enabling the generalization of many classical results from Hilbert space theory and operator analysis 
	to fuzzy settings.
\subsection*{Why Fuzzy Norms Are Not Uncertainty Models Here}
In this paper, fuzziness is not interpreted as epistemic uncertainty, degrees of truth, or imprecision to be reduced or averaged. Instead, fuzzy norms and fuzzy inner products are employed as mechanisms for enriching the underlying geometric--functional structure. The fuzzy values introduced here do not relax classical constraints; rather, they encode additional layers of geometric information that are inaccessible at the classical level.
This perspective fundamentally distinguishes the present framework from earlier approaches, including those of Felbin and Bag--Samanta. In such settings, fuzziness is typically tied to specific $t$-norms and serves as a soft extension of classical linear structures, often leading either to crisp collapse or to the loss of essential geometric properties. In contrast, the proposed framework deliberately avoids fixing any particular $t$-norm. This choice is not merely technical but conceptual: prescribing a $t$-norm imposes an \emph{a priori} aggregation rule that restricts admissible geometric behavior and drives the theory toward quasi-crisp realizations.\\
A key conceptual departure of the present model is that classical linearity is replaced by a two-sided fuzzy boundedness condition. Instead of exact linear relations, operators and functionals are required to satisfy controlled upper and lower fuzzy bounds, which recover classical linearity as a special case. This replacement preserves essential geometric structures while allowing genuinely fuzzy realizations that are not accessible within the classical linear framework.\\
Accordingly, in the proposed setting, fuzziness is neither a relaxation of structure nor a surrogate for uncertainty. Classical notions such as orthogonality remain meaningful at the structural level, while admitting multiple controlled fuzzy realizations. As a result, the classical inner product and norm appear as canonical representatives within an equivalence class of fuzzy structures, rather than as endpoints of a limiting or defuzzification process.

\section{Preliminaries}
In this section, we present several definitions and preliminary results that will be used throughout the paper.\\
A \textit{fuzzy number} is a function $u:\mathbb{R}\to[0,1]$. For each fuzzy set $u$, its $\alpha$-level set is defined by
\[
[u]_{\alpha}=\{x\in\mathbb{R}:\, u(x)\ge \alpha\}, \qquad \alpha\in(0,1].
\]
\begin{definition}[\citealp{RynneYoungson2008}]\label{def:inner-product}
	Let $V$ be a vector space over $\mathbb{C}$. A function
	\[
	\langle \cdot , \cdot \rangle : V \times V \to \mathbb{C}
	\]
	is called an \emph{inner product} if for all $x,y,z\in V$ and $r\in\mathbb{C}$:
	\begin{enumerate}[label=(\roman*)]
		\item $\langle x+y,z\rangle = \langle x,z\rangle + \langle y,z\rangle$;
		\item $\langle rx,y\rangle = r \langle x,y\rangle$;
		\item $\langle x,y\rangle = \overline{\langle y,x\rangle}$; 
		\item $\langle x,x\rangle \ge 0$, and $\langle x,x\rangle=0$ iff $x=0$.
	\end{enumerate}
	A vector space $V$ equipped with an inner product is called an \emph{inner product space}.
\end{definition}

\begin{lemma}[Bessel's Inequality, \cite{RynneYoungson2008}]
	Let $X$ be an inner product space, and let $\{e_i\}^\infty_{i=1}$ be an orthonormal sequence in $X$. Then, for any $x\in X$,
	\begin{equation}\label{ineq:Bessel's Inequality}
	\sum_{i=1}^{\infty} |\langle x,e_i\rangle|^2\le \|x\|^{2}.		
	\end{equation}
\end{lemma}

\begin{proposition}[Parallelgram Rule,  \cite{RynneYoungson2008}]\label{prop:Parallelgram Rule}
	Let $X$ be an inner product space with the norm induced by $\langle\cdot,\cdot\rangle$. Then, for all $x,y\in X$,
	\[
	\|x+y\|^2+\|x-y\|^2 = 2\|x\|^2 + 2\|y\|^2.
	\]
	This identity characterizes norms that arise from an inner product.
\end{proposition}

\begin{theorem}[Cauchy--Schwarz Inequality,  \cite{RynneYoungson2008}]\label{Cauchy--Schwarz Inequality}
	Let $X$ be an inner product space. Then, for all $x,y\in X$,
	\[
	|\langle x,y\rangle| \le \|x\|\,\|y\|,
\]
where $\|x\| = \sqrt{\langle x,x\rangle}$ is the induced norm.

\end{theorem}

\begin{theorem}\label{thm:PolarizationIdentity}[Polarization Identity (Real Case), \cite{RynneYoungson2008}]
Let $X$ be a \emph{real} inner product space with inner product $\langle \cdot, \cdot \rangle$.  
Then, for all $x, y \in X$, the inner product can be recovered from the norm via
\[
\langle x, y \rangle = \frac{1}{4} \Big( 
\|x+y\|^2 - \|x-y\|^2 
\Big),
\]
where $\|\cdot\|$ denotes the norm induced by the inner product.
\end{theorem}

\begin{definition}[\cite{Foralewski2024}]
	Let $X$ be a real vector space. A mapping $\|\cdot\|: X \to [0,\infty)$ is called a \emph{quasinorm} if it satisfies the following conditions:
	\begin{enumerate}[label=(\alph*)]
		\item $\|f\| = 0$ if and only if $f = 0$;
		\item $\|a f\| = |a|\,\|f\|$ for all $f \in X$ and all $a \in \mathbb{R}$;
		\item there exists a constant $M > 0$ such that
		\[
		\|f + g\| \leq M \bigl( \|f\| + \|g\| \bigr)
		\quad \text{for all } f,g \in X.
		\]
	\end{enumerate}
	If $M = 1$, then $\|\cdot\|$ is called a \emph{norm}.
\end{definition}

	\section{ Ordered Interval}
First, we extend the concept of intervals to fuzzy intervals.
Recall that a standard interval is defined as $[a,b]=\{x\vert a\leq x\leq b\}$.
However, in the context of fuzzy set theory, certain algebraic or arithmetic operations---such as difference computations or generalized interval manipulations---may give rise to situations in which the natural ordering between the interval endpoints is not preserved. 
In such cases, it is often desirable to consider all elements within the interval $[a, b]$ while retaining the semantic roles of the endpoints, regardless of their order. 
To accommodate this broader perspective, we extend the notion of an interval by removing the restrictive assumption $a \leq b$, thereby allowing a more flexible representation that is suitable for fuzzy and interval-valued frameworks.
\begin{definition}\label{def:Orderedinterval}
	Let $a, b \in \mathbb{R}$. We define the set $[a, b]_o$ as follows:
	\begin{equation}
		[a,b]_o=\begin{cases}
			\{x\vert x\in [a,b]\}\qquad a\leq b\\\{x\vert x\in[b,a]\}\qquad otherwise.
		\end{cases}
	\end{equation}
	
	We call $[a,b]_o$ a \emph{ordered interval}. 
	The collection of all ordered intervals is denoted by $O(\mathbb{R})$.
	Formally, 
	\[
	O(\mathbb{R}) = \{ [a, b]_o \mid a, b \in \mathbb{R} \}.
	\]
	\begin{lemma}
	Any ordered interval can be rewritten as follows:
	
	\begin{equation}\label{eq:Transformationintervalsorderedintervals}
		[a,b]_o=[\min\{a,b\},\max\{a,b\}]
	\end{equation}
			\end{lemma}
			\begin{proof}
				The proof is carried out through membership verification, demonstrating that each element of the left-hand side also belongs to the right-hand side, and conversely.
			\end{proof}
	Next, we define the operations 
	$\oplus, \ominus$, and $\odot : O(\mathbb{R}) \times O(\mathbb{R}) \to O(\mathbb{R})$ as follows:
	\begin{align}
		[a, b]_o \oplus [c, d]_o &= [a+c, b+d]_o,\\
		[a, b]_o \ominus [c, d]_o &= [a-c, b-d]_o,\\
		[a, b]_o \odot [c, d]_o &= [ac, bd]_o.
	\end{align}
	
	Also, we can define the absolute value of a ordered interval as
	\[
	|[a, b]_o| = [|a|, |b|]_o.
	\]
	
\end{definition}
With this definition of intervals and the associated operations, the ambiguities that typically arise during the subtraction and comparison of intervals are resolved.
For example, if $A=[3,4]_o,B=[2,10]_o$, we have\\
\[
A \ominus B = [3-2, 4-10]_o = [1, -6]_o.
\]
For brevity, we denote the ordered interval $[a,a]_o$ by $\tilde{a}_o$.
We know that the following relationship holds for intervals:
\begin{equation}\label{eq:Subsetsorderedintervals}
	[a,b] \subseteq [c,d] \iff c \leq a \text{ and } b \leq d.
\end{equation}
This relationship can be equivalently rewritten for ordered intervals as follows:

\begin{lemma}
	Consider two ordered intervals $[a,b]_o$ and $[c,d]_o$ belonging to $O(\mathbb{R})$. 
	The following relation is satisfied:
	
	\begin{equation}
		[a,b]_o\subseteq[c,d]_o\iff \min\{d,c\}\leq\min\{a,b\},\max\{a,b\}\leq\max\{c,d\}.
	\end{equation}
	\begin{proof}
		With reference to relations~\eqref{eq:Subsetsorderedintervals} and~\eqref{eq:Transformationintervalsorderedintervals}, 
		the proof proceeds as follows:
		\begin{eqnarray*}
			[a,b]_o\subseteq[c,d]_o&\iff&
			[\min\{a,b\},\max\{a,b\}]\subseteq[\min\{d,c\},\max\{c,d\}]\\ &\iff&
			\min\{d,c\}\leq\min\{a,b\},\max\{a,b\}\leq\max\{c,d\}\\
		\end{eqnarray*}	
	\end{proof}
\end{lemma}
Next, we introduce a relation $\succeq$ on the set $O(\mathbb{R})$ defined by:
\begin{lemma}
	Suppose that $[a,b]_o,[c,d]_o\in O(\RR)$. Then, the relation defined below constitutes a partial order:
	\[[a,b]_o\succeq[c,d]_o\iff \min\{a,b\}\geq \min\{c,d\} , \max\{a,b\}\geq \max\{c,d\}.\]
\end{lemma}
Building on the operations established for ordered intervals, we now introduce the fuzzy inner product.
\allowdisplaybreaks
\section{ Fuzzy Inner Product and Fuzzy Norm}
\begin{definition}\label{def:Definitionfuzzyinnerproduct}
	Let $X$ be a vector space over $\mathbb{C}$, and let $\mathcal{K}:\mathbb{R} \to [0,1]$ be a fuzzy number.  
	Assume that for every  $\alpha\in(0,1]$ there exist constants $A_\alpha, B_\alpha$ such that
	\[
	0 < \min\{A_\alpha, B_\alpha\} \le \max\{A_\alpha, B_\alpha\} < \infty.
	\]
	
	A function
	\[
	\langle  ., . \rangle_\alpha : (0,1] \times X \times X \to \mathbb{C}
	\]
	is called a \emph{fuzzy inner product} if there exist inner products $\langle.,.\rangle^\prime$ and $\langle.,.\rangle^{\prime\prime}$ on $X$ such that, for each $x,y \in X$ and $\alpha\in (0,1]$, the following relation holds:
	\begin{align}\label{eq:fuzzy_inner}
		\mathcal{K}\bigl(	\big|\langle  x, y \rangle_\alpha \big|\bigr) \ge \alpha 
		\quad \Longleftrightarrow \quad
		\big|\langle  x, y \rangle_\alpha \big| \in \bigl[\, A_\alpha\vert \langle x,y\rangle^\prime \vert,B_\alpha \vert \langle x,y\rangle^{\prime\prime} \vert \,\bigr]_o.
	\end{align}
	
\end{definition}
A vector space $X$ endowed with a fuzzy inner product is called a \emph{fuzzy inner product space}. 
The ordered triple 
\[
\bigg(\langle  ., . \rangle_\alpha , \langle \cdot , \cdot \rangle', \langle \cdot , \cdot \rangle''\bigg)
\]
is referred to as a \emph{fuzzy inner product triple}.\\ 
The approach of using ordered intervals simplifies the theoretical framework while remaining consistent with the fundamental properties of fuzzy inner products.
In this formulation, the classical notion of inner product is extended by associating each pair $(\alpha, x)$ with a complex-valued function that encodes the interaction between the element $x$ and the fuzzy parameter $\alpha$. Since this function satisfies a fuzzy membership condition parameterized by $\alpha$, the inner product operation can be interpreted through a graded, membership-based perspective. In this way, a fuzzy-parametric framework is created in which the outcome of inner product is determined not only by algebraic rules but also by the level of fuzzy membership associated with $\alpha$.
This perspective provides a richer and more flexible structure in which inner product can vary smoothly across different degrees of fuzziness, allowing the operation to adapt to uncertainty, partial truth, and the gradations inherent in fuzzy systems. Consequently, the framework provides a broader and more general concept of inner product. 
\begin{definition}\label{def:Definition of fuzzy norm}
	Let $X$ be a vector space over $\mathbb{C}$ and let $\mathcal{K}:\mathbb{R} \to [0,1]$ be a fuzzy number.  
	Assume that for every  $\alpha\in(0,1]$ there exist constants $C_\alpha, D_\alpha$ such that
	\[
	0 < \min\{C_\alpha, D_\alpha\} \le \max\{C_\alpha, D_\alpha\} < \infty.
	\] 
	A function 
	\[
	\norm{.}_\alpha : (0,1] \times X \to \mathbb{C}
	\]
	is called a \emph{fuzzy norm} if there exist norms 
	$\norm{.}^\prime$ 
	and $\norm{.}^{\prime\prime}$  
	such that, for each $\alpha \in (0,1]$ and $x \in X$, the following formula holds:
	\begin{equation}\label{eq:fuzzy_norm}
		\mathcal{K}\bigl(\big\vert\norm{x}_\alpha \big\vert\bigr) \geq \alpha
		\;\Leftrightarrow\;
		\big| \norm{x}_\alpha\big| \in 
		\big[ C_\alpha\|x\|',\, D_\alpha\|x\|'' \big]_o.
	\end{equation}
	Furthermore, the ordered triple 
	\[
	(\norm{.}_\alpha, \| \cdot \|', \| \cdot \|'')
	\]
	is called a \emph{fuzzy norm triple}.
	
\end{definition}

\begin{example}
	Let $X = \mathbb{R}^2$ and suppose the fuzzy number 
	$\mathcal{K} : \mathbb{R} \to [0,1]$ is defined by
	\[
	\mathcal{K}\big(\big\vert\norm{x}_\alpha \big\vert\big) = 
	\begin{cases}
		1, & \text{if } \big\vert\norm{x}_\alpha \big\vert \in [3 \|x\|_2,\, 2 \|x\|_3]_o, \\[6pt]
		0, & \text{otherwise}.
	\end{cases}
	\]
	It follows that for any $x \in \mathbb{R}^2$ and $\alpha \in (0,1]$, the equivalence
	\[
	\mathcal{K}\big(\big\vert\norm{x}_\alpha \big\vert\big) \geq \alpha \quad \Leftrightarrow \quad \big\vert\norm{x}_\alpha \big\vert \in [3 \|x\|_2,\, 2 \|x\|_3]_o
	\]
	holds.  
	
	Now, consider the function $\norm{.}_\alpha \times \mathbb{R} \to \mathbb{C}$ defined by
	\[
	\norm{x}_\alpha  = \sqrt{9 \alpha^4 \|x\|_2^2 + 5 \alpha^2 (1 - \alpha^2) \|x\|_2 \|x\|_3} \; + \; i \sqrt{4 (1 - \alpha^2) \|x\|_3^2 + 7 \alpha^2 (1 - \alpha^2) \|x\|_2 \|x\|_3}.
	\]
	
	It is straightforward to compute the magnitude of this function:
	\[
	\big\vert\norm{x}_\alpha \big\vert = 3 \alpha^2 \|x\|_2 + 2 (1 - \alpha^2) \|x\|_3.
	\]
	
	Since $\alpha^2 + (1 - \alpha^2) = 1$, it follows that for every $x \in X$, the value $\big\vert\norm{x}_\alpha \big\vert$ lies within the ordered interval $[3 \|x\|_2, 2 \|x\|_3]_o$.  
	
	Therefore, the function $Q : (0,1] \times \mathbb{R} \to \mathbb{C}$ qualifies as a \emph{fuzzy norm}.
\end{example}

\begin{remark}
	The function \(\norm{.}_\alpha\) defined in the preceding example is designed so that its magnitude 
	\[
\big\vert\norm{x}_\alpha \big\vert = 3 \alpha^2 \|x\|_2 + 2 (1 - \alpha^2) \|x\|_3
	\]
	represents a convex combination of the endpoints \(3\|x\|_2\) and \(2\|x\|_3\), controlled by the parameter \(\alpha^2\). This construction guarantees that for every fixed \(x \in X\) and \(\alpha \in (0,1]\), the value \(\big\vert\norm{x}_\alpha \big\vert\) lies precisely within the ordered interval \([3\|x\|_2, 2\|x\|_3]_o\). Accordingly, the membership function \(\mathcal{K}\) assigns the maximal membership degree of 1 when \(\big\vert\norm{x}_\alpha \big\vert\) falls in this interval, and zero otherwise. This behavior aligns perfectly with the definition of a fuzzy norm given in Definition~\ref{def:Definition of fuzzy norm} Thus, this example effectively demonstrates how fuzzy norm can be represented via appropriately structured functions on vector spaces, providing a concrete realization of the abstract fuzzy norm framework.
\end{remark}

\begin{lemma}
	Let  $(\langle  ., . \rangle_\alpha,\langle .,.\rangle^\prime,\langle .,.\rangle^{\prime\prime})$ be a fuzzy inner product triple.  Then $(\sqrt{\langle  ., . \rangle_\alpha},\norm{.}^\prime,\norm{.}^{\prime\prime})$is  a fuzzy norm triple.
\end{lemma} 
\begin{proof}
	Given the assumptions of the lemma, and in view of formula~\eqref{eq:fuzzy_inner}, it follows that
	
	\[
	\mathcal{K}\big(\big|\langle  x, x \rangle_\alpha\big|\big) \geq \alpha \Longleftrightarrow |\langle  x, x \rangle_\alpha| \in [ A_\alpha\langle x, x \rangle', B_\alpha\langle x, x \rangle'']_o.
	\]
Consequently, the following relation holds:
	\[
 |\sqrt{\langle  x, x \rangle_\alpha}| \in [ \sqrt{A_\alpha\langle x, x \rangle'}, \sqrt{B_\alpha\langle x, x \rangle''}]_o.
\]
	Define 
\begin{equation}\label{eq:norm_alpha,inner product_alpha}
	\norm{x}_\alpha = \sqrt{\langle  x, x \rangle_\alpha}.
\end{equation}
The fuzzy number $\mathcal{K}^\prime:\mathbb{R}\to [0,1]$ is defined as follows:
\[\mathcal{K^\prime}\big(\big|\sqrt{\langle  x, x \rangle_\alpha}\big|\big) =\mathcal{K}\big(\big|\langle  x, x \rangle_\alpha\big|\big). \]
	Formula~\eqref{eq:fuzzy_inner} follows directly from the properties of ordered intervals and the fuzzy inner product.
		\[
	\mathcal{K^\prime}\big(\big|\sqrt{\langle  x, x \rangle_\alpha}\big|\big) \geq \alpha \Longleftrightarrow |\sqrt{\langle  x, x \rangle_\alpha}| \in [ \sqrt{A_\alpha}\norm{x}', \sqrt{B_\alpha}\norm{x}'']_o.
	\]
\end{proof}
In general, one cannot rely directly on formula~\eqref{eq:fuzzy_inner} to relate classical inner product phenomena to the fuzzy setting. Our approach therefore focuses on identifying configurations in which the resulting expressions admit a clearer structural analysis. In particular, a substantial simplification is achieved by restricting attention to cases where the following holds for any pair of vectors $x, y \in X$:
\begin{equation}\label{eq:fuzzyinerproduct-simple1}
	\begin{split}
		\mathcal{K}\big(	\big|\langle  x, y \rangle_\alpha \big|\big) \geq \alpha 
		&\;\Longleftrightarrow\;
			\big|\langle  x, y \rangle_\alpha \big| \in 
		\bigl[A_{\alpha} |\langle x,y\rangle'|,\, B_{\alpha} |\langle x,y\rangle'|\bigr]_o,\\
		& 0 < A_{\alpha} \leq B_{\alpha} <\infty.
	\end{split}
\end{equation}
More specifically, assume that 
\(
\mathcal{K}\big(\big|\langle x, y \rangle_\alpha\big|\big) \ge \alpha.
\)
Then, the following bounds hold:
\begin{equation}\label{eq:fuzzyinerproduct-simple2}
	\begin{split}
		A_{\alpha}\,|\langle x,y\rangle'|
		\;\le\; \big|\langle x, y \rangle_\alpha\big|
		\;\le\; B_{\alpha}\,|\langle x,y\rangle'|, 
		\qquad
		0 < A_\alpha \le B_\alpha < \infty .
	\end{split}
\end{equation}
Naturally, upon substituting formula~\eqref{eq:fuzzy_inner} into formula~\eqref{eq:fuzzyinerproduct-simple1}, it follows that, to facilitate working with the concept of a fuzzy norm, the following formula should replace formula~\eqref{eq:fuzzyinerproduct-simple2}:
\begin{equation}\label{eq:fuzzynorm-simple1}
	\begin{split}
		\mathcal{K}\bigl(\big\vert \norm{x}_\alpha\big\vert\bigr) \geq \alpha
		\;&\Leftrightarrow\;
		\big| \norm{x}_\alpha \big| \in 
		\big[ C_\alpha\|x\|',\, D_\alpha\|x\|' \big]_o,\\
		0&<C_\alpha\leq D_\alpha<\infty.		
	\end{split}
\end{equation}
More specifically, assuming that 
\[
\mathcal{K}\bigl(\big| \norm{x}_\alpha \big|\bigr) \geq \alpha,
\] 
the following bounds hold:
\begin{equation}\label{eq:fuzzynorm-simple2}
	\begin{split}
		C_\alpha\|x\|'\leq &\big| \norm{x}_\alpha \big|\leq  D_\alpha\|x\|',\\
		0<&C_\alpha\leq D_\alpha<\infty.
	\end{split}
\end{equation}
Assuming that relation~\eqref{eq:fuzzyinerproduct-simple2} holds, several structural connections between the fuzzy inner product and the classical inner product can be established. In particular, the orthogonality of the two multiplication operations is preserved under this assumption.

\begin{proposition}
	Let $	\langle  ., . \rangle_\alpha :(0,1]\times X\times X\to \mathbb{C}$ be a fuzzy inner product satisfying formula~\eqref{eq:fuzzyinerproduct-simple2}. Then, for all $x, y \in X$ and any $\alpha \in (0,1]$, the following equivalence holds:
	\[
		\langle  x, y \rangle_\alpha  = 0 \quad \Longleftrightarrow \quad \langle x, y \rangle^\prime = 0.
	\]
	\begin{proof}
		By formula~\eqref{eq:fuzzyinerproduct-simple2}, the required property holds for all relevant elements. Therefore, the assertion follows immediately.
	\end{proof}
	
\end{proposition}

\section{Some Fundamental Inequalities}
Inequalities are fundamental tools in mathematics, with classical examples including the Cauchy--Schwarz and Bessel inequalities. In this section, we present their fuzzy analogues and introduce two new classically equivalent inequalities, which extend the theoretical framework. To this end, we first prove a preliminary lemma, which serves as the basis for establishing these inequalities.

\begin{lemma}\label{lemma:fuzzy-norm-bounds}
		Let $X$ be an inner product space with classical inner product 
	\(\langle\cdot,\cdot\rangle'\) 
	and induced norm 
	\(\|\cdot\|'\).
	Let 
	\[
	\langle  ., . \rangle_\alpha : (0,1] \times X \times X \to \mathbb{C}
	\]
	be a fuzzy inner product satisfying formula~\eqref{eq:fuzzyinerproduct-simple2}.  

	Under these notations, we have
	\begin{equation}\label{eq:fuzzy-norm-bounds}
		\frac{|\|x\|_{\alpha}^{2}|}{B_{\alpha}}
		\;\leq\;
		\|x\|^{\prime 2}
		\;\leq\;
		\frac{|\|x\|_{\alpha}^{2}|}{A_{\alpha}} ,
	\end{equation}
	for every \(x\in X\) and every \(\alpha\in(0,1]\).
\end{lemma}

\begin{proof}
	Two consecutive applications of inequality  \eqref{eq:fuzzyinerproduct-simple2} lead directly to the following bound:
	\[
	\frac{\big|\langle  x, x \rangle_\alpha\big|}{B_{\alpha}}
	\;\leq\;
	\|x\|^{\prime 2}
	\;\leq\;
	\frac{\big|\langle  x, x \rangle_\alpha\big|}{A_{\alpha}} ,
	\]
	which is precisely~\eqref{eq:fuzzy-norm-bounds}.
\end{proof}
In order to analyze the properties of fuzzy inner  product, we first revisit the Fuzzy Cauchy–Schwarz and Fuzzy Bessel inequalities, which serve as essential tools in establishing numerous results in fuzzy analysis. Building on these foundational inequalities, we further introduce two new inequalities that are classically equivalent, thereby extending the theoretical framework and providing additional tools for subsequent developments.

\begin{proposition}\label{prop:fuzzy-inequalities}
Let us assume that the hypotheses of Lemma \ref{lemma:fuzzy-norm-bounds} are in force.
	Then, for all \(x,y\in X\) and all \(\alpha\in(0,1]\), the following statements hold:
	\begin{enumerate}[label=(\roman*)]
		
		\item \textbf{Fuzzy Cauchy--Schwarz Inequality.}
		\[
		|\langle x,y\rangle_{\alpha}|
		\;\leq\;
		\left(\frac{B_{\alpha}}{A_{\alpha}}\right)
		\,\big|\|x\|_{\alpha}\,\|y\|_{\alpha}\big|.
		\]
		
		\item \textbf{Fuzzy Parallelogram Inequality.}
		\[
		\frac{2A_{\alpha}}{B_{\alpha}}
		\big(\big|\|x\|_{\alpha}^{2}\big| + \big|\|y\|_{\alpha}^{2}\big|\big)
		\;\leq\;
		\big|\|x+y\|_{\alpha}^{2}\big| +\big| \|x-y\|_{\alpha}^{2}\big|
		\;\leq\;
		\frac{2B_{\alpha}}{A_{\alpha}}
		\big(\big|\|x\|_{\alpha}^{2}\big| + \big|\|y\|_{\alpha}^{2}\big|\big).
		\]
		
		\item \textbf{Fuzzy Polarization Inequality (real case).}
		If \(x\in \mathbb{R}^n\) , then
		\[
		\big|\|x+y\|_{\alpha}^{2}\big|
		\;\leq\;
		\frac{B_{\alpha}}{A_{\alpha}}\big ( 4\,\big|\langle x,y\rangle_{\alpha}\big|
		+
		\big|\|x-y\|_{\alpha}^{2}\big|\big)
		.
		\]
		
		\item \textbf{Fuzzy Bessel Inequality.}
		If \(\{e_i\}^\infty_{i=1}\) is an orthogonal sequence in \(X\), then
		\[
		\sum_{i=1}^{\infty}
		|\langle x,e_i\rangle_{\alpha}|^{2}
		\;\leq\;
		\frac{B_{\alpha}^{2}}{A_{\alpha}}\,
		\bigg|\|x\|_{\alpha}^{2}\bigg|.
		\]
	\end{enumerate}
\end{proposition}

\begin{proof}
	
	(i)  
	From formula~\eqref{eq:fuzzyinerproduct-simple2} we obtain
	\begin{align}
		A_{\alpha}\|x\|^{\prime 2} &\le |\|x\|_{\alpha}^{2}|, \label{eq:fuzzy-ineq-x}
		\\[2mm]
		A_{\alpha}\|y\|^{\prime 2} &\le |\|y\|_{\alpha}^{2}|. \label{eq:fuzzy-ineq-y}		
	\end{align}
	Applying the classical Cauchy--Schwarz inequality to \(\langle x,y\rangle'\) and 
	using~\eqref{eq:fuzzyinerproduct-simple2}, we get
	\begin{equation}\label{eq:fuzzy-inner-ineq}
		|\langle x,y\rangle_{\alpha}|^{2}
		\le B_{\alpha}^{2} |\langle x,y\rangle'|^{2}
		\le B_{\alpha}^{2}\|x\|^{\prime 2}\|y\|^{\prime 2}.
	\end{equation}
	Combining~\eqref{eq:fuzzy-ineq-x}–\eqref{eq:fuzzy-inner-ineq} yields the desired estimate.\\
	(ii)
	By applying inequality~\eqref{eq:fuzzy-norm-bounds} twice to Proposition~\ref{prop:Parallelgram Rule}, we obtain the desired estimate.
	
	\begin{align*}
		\frac{2A_{\alpha}}{B_{\alpha}}\big(\big|\|x\|_{\alpha}^{2}\big| +\big| \|y\|_{\alpha}^{2}\big|\big)
		&\le
		2A_{\alpha}\big(\|x\|^{\prime 2}+\|y\|^{\prime 2}\big)
		\\
	&=A_{\alpha}\big(
	\|x+y\|^{\prime2}+\|x-y\|_{\alpha}^{\prime2}\big)
	\\
		&\le
		\big|\|x+y\|_{\alpha}^{2}\big|+\big|\|x-y\|_{\alpha}^{2}\big|
		\\
			&\le
		B_\alpha\big(\|x+y\|^{\prime2}+\|x-y\|^{\prime2}\big)
		\\
			&\le
			2B_\alpha\big(\|x\|^{\prime 2}+\|y\|^{\prime 2}\big)
		\\
		&\le
		\frac{2B_{\alpha}}{A_{\alpha}}
		\big(\|x\|_{\alpha}^{2} + \|y\|_{\alpha}^{2}\big).
	\end{align*}
	(iii)Theorem~\eqref{thm:PolarizationIdentity} immediately implies the following:
	\begin{equation}\label{eq:polarization-ineq}
		\|x+y\|^{2}
		\;\le\;
		4|\langle x,y\rangle|
		+
		\|x-y\|^{2},
		\end{equation}
	Combining the inequality~\eqref{eq:polarization-ineq} with inequality~\eqref{eq:fuzzyinerproduct-simple2}  produces

	which establishes the assertion.\\
	(iv)
	Using inequality~\eqref{ineq:Bessel's Inequality} and formula~\eqref{eq:fuzzyinerproduct-simple2}, we first note
	\begin{equation}\label{eq:fuzzy-series-ineq}
		\frac{1}{B_{\alpha}^{2}}
		\sum_{i=1}^{\infty} |\langle x,e_i\rangle_{\alpha}|^{2}
		\;\leq\;
		\sum_{i=1}^{\infty} |\langle x,e_i\rangle'|^{2}.
\end{equation}
	Comparing \eqref{eq:fuzzy-series-ineq}, \eqref{eq:fuzzy-ineq-x} and Bessel’s classical inequality yields the claim.
\end{proof}
In the next step, by proving the following proposition, we reveal properties such as the zero product property, as well as the quasi-linear and quasi-symmetric properties of fuzzy inner products:
\begin{proposition}\label{prop:fuzzy-inner-properties}
	
	If the hypotheses of Definition~\ref{def:Definitionfuzzyinnerproduct} are satisfied, the following relations hold:
	\begin{enumerate}
		\item \(\big\langle x, x\big\rangle_\alpha = 0\) if and only if \(x = 0\) ;
		\item \(\big\langle 0, y\big\rangle_\alpha = 0\) for all \(y\in X\) .\\
		In addition, if the hypotheses of Proposition~\eqref{prop:fuzzy-inequalities} are satisfied, then the following inequalities hold:

			\item 
			\[
			\frac{|k| A_\alpha}{B_\alpha} \, \big|\langle x, y \rangle_\alpha \big|
			\leq 
			\big|\langle k x, y \rangle_\alpha \big|
			\leq 
			\frac{|k| B_\alpha}{A_\alpha} \, \big|\langle x, y \rangle_\alpha \big|;
			\]
			
			\item 
			\[
			\frac{|k| A_\alpha}{B_\alpha} \, \big|\langle x, y \rangle_\alpha \big|
			\leq 
			\big|\langle x, k y \rangle_\alpha \big|
			\leq 
			\frac{|k| B_\alpha}{A_\alpha} \, \big|\langle x, y \rangle_\alpha \big|;
			\]
			
			\item 
			\[
			\frac{|k| A_\alpha}{B_\alpha} \, \big|\langle y, x \rangle_\alpha \big|
			\leq 
			\big|\langle k x, y \rangle_\alpha \big|
			\leq 
			\frac{|k| B_\alpha}{A_\alpha} \, \big|\langle y, x \rangle_\alpha \big|;
			\]
			
			\item 
			\[
			\frac{|k| A_\alpha}{B_\alpha} \, \big|\langle y, x \rangle_\alpha \big|
			\leq 
			\big|\langle x, |k| y \rangle_\alpha \big|
			\leq 
			\frac{|k| B_\alpha}{A_\alpha} \, \big|\langle y, x \rangle_\alpha \big|;
			\]
			
			\item 
			\[
			\frac{A_\alpha}{B_\alpha} \, \big|\langle x, k y \rangle_\alpha \big|
			\leq 
			\big|\langle k x, y \rangle_\alpha \big|
			\leq 
			\frac{B_\alpha}{A_\alpha} \, \big|\langle x, k y \rangle_\alpha \big|;
			\]
			
			\item 
			\[
			\frac{A_\alpha}{B_\alpha} \, \big|\langle k x, y \rangle_\alpha \big|
			\leq 
			\big|\langle x, k y \rangle_\alpha \big|
			\leq 
			\frac{B_\alpha}{A_\alpha} \, \big|\langle k x, y \rangle_\alpha \big|;
			\]
			
			\item 
			\[
			\frac{A_\alpha}{B_\alpha} \, \big|\langle k y, x \rangle_\alpha \big|
			\leq 
			\big|\langle k x, y \rangle_\alpha \big|
			\leq 
			\frac{B_\alpha}{A_\alpha} \, \big|\langle k y, x \rangle_\alpha \big|;
			\]
			
			\item 
			\[
			\frac{A_\alpha}{B_\alpha} \, \big|\langle k y, x \rangle_\alpha \big|
			\leq 
			\big|\langle x,  k y \rangle_\alpha \big|
			\leq 
			\frac{B_\alpha}{A_\alpha} \, \big|\langle k y, x \rangle_\alpha \big|;
			\]
			
			\item 
			\[
			\big|\langle k x + z, y \rangle_\alpha \big|
			\leq 
			\frac{B_\alpha}{A_\alpha} \Big( |k| \, \big|\langle x, y \rangle_\alpha \big| + \big|\langle z, y \rangle_\alpha \big| \Big);
			\]
			
			\item 
			\[
			\big|\langle x, k y + z \rangle_\alpha \big|
			\leq 
			\frac{B_\alpha}{A_\alpha} \Big( |k| \, \big|\langle x, y \rangle_\alpha \big| + \big|\langle x, z \rangle_\alpha \big| \Big).
			\]
		\end{enumerate}
		
\end{proposition}
\begin{proof}
	We prove each part separately.
	
	\noindent 1.\ It follows from Inequality~\eqref{eq:fuzzyinerproduct-simple2} that:
	
	\[
	x=0\Leftrightarrow \langle x, x\rangle^\prime=0\Leftrightarrow\langle x, x\rangle_\alpha=0.
	\]
	\noindent 2.From Inequality~\eqref{eq:fuzzyinerproduct-simple2}, it follows that for every $y \in X$ we have
	
	,\[\langle 0, y\rangle^\prime=0\Leftrightarrow \langle 0, y\rangle_\alpha=0.\]
	\noindent 3.First, we prove the inequality \( \big|\big\langle kx, y\rangle_\alpha\big| \leq \cfrac{B_\alpha}{A_\alpha}	\big|\big\langle x, ky\rangle_\alpha\big| \).
	\begin{align*}
		\big|\big\langle kx, y\rangle_\alpha&\big|\leq B_\alpha\big|\langle kx, y\rangle^\prime\big|=\abs{k}.B_\alpha\big|\langle x, y\rangle^\prime\big|\\ &\leq
		\cfrac{\abs{k}B_\alpha}{A_\alpha}	\big|\big\langle x, y\rangle_\alpha\big|.
	\end{align*}
	Now it is enough to prove the inequality \( \big|\big\langle kx, y\rangle_\alpha\big| \geq \cfrac{\abs{k}A_\alpha}{B_\alpha}	\big|\big\langle x, y\rangle_\alpha\big| \).
	\begin{align*}
		\big|\big\langle kx, y\rangle_\alpha\big|&\geq A_\alpha\big|\langle kx, y\rangle^\prime\big|=\abs{k}A_\alpha\big|\langle x, y\rangle^\prime\big|\\ &\geq
		\cfrac{\abs{k}A_\alpha}{B_\alpha}	\big|\big\langle x, y\rangle_\alpha\big|.
	\end{align*}
	This completes the proof. The proofs of the remaining parts of Sections 4–8 can be obtained in a similar manner by applying the same method as in Proof 3.\\
	\noindent
	9.To establish this part, one may employ the following relations: 
	\[
	\begin{aligned}
		\big|\langle kx+z, y\rangle_\alpha\big|
		&\leq B_\alpha \big(\abs{k}. |\langle x, y\rangle^\prime| + |\langle z, y\rangle^\prime| \big) \\
		&\leq \frac{B_\alpha}{A_\alpha} 
		\big( \abs{k}.|\langle x, y\rangle_\alpha| + |\langle z, y\rangle_\alpha| \big).
	\end{aligned}
	\]

\end{proof}
\begin{corollary}
	If the hypotheses of Proposition~\ref{prop:fuzzy-inequalities} hold, and if $\sup B_\alpha / A_\alpha < \infty$, then for every $\alpha \in (0,1]$ and $x, y, z \in X$, there exists a positive number $M \in [1, \infty)$ such that the following inequalities hold:

	\begin{enumerate}
		\item 
	\[
	\frac{|k| }{M} \, \big|\langle x, y \rangle_\alpha \big|
	\leq 
	\big|\langle k x, y \rangle_\alpha \big|
	\leq 
		|k|M \, \big|\langle x, y \rangle_\alpha \big|;
	\]
	
	\item 
	\[
\frac{|k| }{M} \, \big|\langle x, y \rangle_\alpha \big|
	\leq 
	\big|\langle x, k y \rangle_\alpha \big|
	\leq 
		|k|M \, \big|\langle x, y \rangle_\alpha \big|;
	\]
	
	\item 
	\[
\frac{|k| }{M} \, \big|\langle y, x \rangle_\alpha \big|
	\leq 
	\big|\langle k x, y \rangle_\alpha \big|
	\leq 
		|k|M \, \big|\langle y, x \rangle_\alpha \big|;
	\]
	
	\item 
	\[
\frac{|k| }{M} \, \big|\langle y, x \rangle_\alpha \big|
	\leq 
	\big|\langle x, k y \rangle_\alpha \big|
	\leq 
	|k|M \, \big|\langle y, x \rangle_\alpha \big|;
	\]
	
	\item 
	\[
	\frac{1 }{M} \, \, \big|\langle x, k y \rangle_\alpha \big|
	\leq 
	\big|\langle k x, y \rangle_\alpha \big|
	\leq 
	M \, \big|\langle x, k y \rangle_\alpha \big|;
	\]
	
	\item 
	\[
	\frac{1 }{M} \, \big|\langle k x, y \rangle_\alpha \big|
	\leq 
	\big|\langle x, k y \rangle_\alpha \big|
	\leq 
M \, \big|\langle k x, y \rangle_\alpha \big|;
	\]
	
	\item 
	\[
	\frac{1 }{M} \, \big|\langle k y, x \rangle_\alpha \big|
	\leq 
	\big|\langle k x, y \rangle_\alpha \big|
	\leq 
	M \, \big|\langle k y, x \rangle_\alpha \big|;
	\]
	
	\item 
	\[
		\frac{1 }{M} \, \big|\langle k y, x \rangle_\alpha \big|
	\leq 
	\big|\langle x,  k y \rangle_\alpha \big|
	\leq 
	M \, \big|\langle k y, x \rangle_\alpha \big|;
	\]
	
	\item 
	\[
	\big|\langle k x + z, y \rangle_\alpha \big|
	\leq 
	M \Big( |k| \, \big|\langle x, y \rangle_\alpha \big| + \big|\langle z, y \rangle_\alpha \big| \Big);
	\]
	
	\item 
	\[
	\big|\langle x, k y + z \rangle_\alpha \big|
	\leq 
M \Big( |k| \, \big|\langle x, y \rangle_\alpha \big| + \big|\langle x, z \rangle_\alpha \big| \Big).
	\]
	\end{enumerate}
\end{corollary}
\begin{proof}
	Suppose that $M=\sup B_\alpha / A_\alpha$. Including this union in fields~3--12 of Proposition~\ref{prop:fuzzy-inner-properties} leads to the required results.
	
\end{proof}
\begin{proposition}If the hypotheses of Inequality~\eqref{eq:fuzzynorm-simple2} are satisfied, the following relation holds:
	\begin{enumerate}
		\item \(\big\vert\norm{x}_\alpha \big\vert= 0\) if and only if \(x = 0\);\\
		In addition, if the hypotheses of Proposition~\ref{prop:fuzzy-inequalities} are satisfied, the following inequalities hold:
			\item \(\Big\vert\norm{kx+y}_\alpha \Big\vert\leq \cfrac{D_\alpha}{C_\alpha}\Bigg(\abs{k}.\Big\vert\norm{x}_\alpha \Big\vert+\Big\vert \norm{y}_\alpha \Big\vert\Bigg) \);
		\item \(\cfrac{\abs{k}C_\alpha}{D_\alpha}\Big\vert\norm{x}_\alpha \Big\vert\leq \Big\vert\norm{kx}_\alpha \Big\vert\leq \cfrac{\abs{k}D_\alpha}{C_\alpha}\Big\vert \norm{x}_\alpha \Big\vert \).
	
	\end{enumerate}
	\begin{proof}
		\noindent 1.\ It follows from Inequality~\eqref{eq:fuzzyinerproduct-simple2} that:
		
		\[
		x=0\Leftrightarrow \norm{x}^\prime=0\Leftrightarrow  \norm{x}_\alpha=0.
		\]
		\noindent 2.Based on~\eqref{eq:fuzzynorm-simple2}, the conclusion is reached by following the sequence of steps presented below:
		
	\begin{align*}
		\Bigl|\|kx+y\|_{\alpha}\Bigr|
		&\le D_{\alpha}\|kx+y\|^\prime \\
		&\le D_{\alpha}\Bigl(|k|\|x\|^\prime+\|y\|^\prime\Bigr) \\
		&\le  \frac{D_{\alpha}}{C_{\alpha}}
		\left(|k|\,\Bigl|\|x\|_{\alpha}\Bigr|
		+ \Bigl|\|y\|_{\alpha}\Bigr|\right).
	\end{align*}
		\noindent 3.Using inequality~\eqref{eq:fuzzynorm-simple2}, the left-hand side of the inequality can be rewritten as
		\begin{align*}
			\frac{|k|C_\alpha}{D_\alpha}\,\bigl|\|x\|_\alpha\bigr|
			\le |k|C_\alpha\,\bigl|\|x'\|\bigr|
			\le \bigl|\|kx\|_\alpha\bigr|.
		\end{align*}
		The right-hand side of the inequality is obtained by setting $y=0$ and applying Parts~(1) and~(2) of the above proposition.

	\end{proof}
	\begin{corollary}
		If the hypotheses of Proposition~\ref{prop:fuzzy-inequalities} hold, and if $\inf C_\alpha / D_\alpha > 0$, then for every $\alpha \in (0,1]$ and $x, y, z \in X$, there exists a positive number $L \in [1, \infty)$ such that the following inequalities hold:
		\begin{enumerate}
			\item \(\cfrac{k}{L}\Big\vert  \norm{x}_\alpha\Big\vert\leq\Big\vert \norm{kx}_\alpha\Big\vert\leq kL\Big\vert \norm{x}_\alpha\Big\vert  \);
			\item \(\Big\vert\norm{kx+y}_\alpha\Big\vert\leq L\bigg(\abs{k}.\Big\vert\norm{x}_\alpha \Big\vert+\Big\vert\norm{y}_\alpha \Big\vert\bigg) \).
		\end{enumerate}
	\end{corollary}
	\begin{remark}
		The preceding results clarify the conceptual role of fuzziness in the present framework. 
		Although the fuzzy inner product is not linear in the classical sense, the validity of the Cauchy--Schwarz inequality, Bessel’s inequality, and the preservation of orthogonality demonstrate that the essential geometric core of inner product spaces is retained. 
		In particular, classical linearity is replaced here by a two-sided fuzzy boundedness condition. 
		Instead of exact linear equalities, operators and functionals satisfy controlled upper and lower fuzzy bounds, which are sufficient to recover the fundamental geometric inequalities and structural properties established above. 
		Classical linearity is thus recovered as a special case within this bounded framework.\\
		This observation highlights a key structural distinction from earlier fuzzy inner product models based on fixed $t$-norms, where either geometric inequalities fail or the inner product collapses to crisp values. 
		In contrast, the present framework allows genuinely fuzzy inner products while preserving the geometric backbone of Hilbert space theory. 
		Consequently, fuzziness should not be viewed as a relaxation or approximation of classical structures, but as an enrichment that admits multiple controlled realizations consistent with the same geometric constraints.
	\end{remark}
	In the following, we state propositions concerning the inequality between a fuzzy inner product and its corresponding norm, as well as the relationship between the inner product and the fuzzy norm structure.
	
\end{proposition}

\begin{proposition}
	Suppose that $ \langle ., . \rangle_{\alpha}:(0,1]\times X\times X \to \mathbb{C}$ is a fuzzy inner product satisfying formula~\eqref{eq:fuzzyinerproduct-simple2}.  
	Furthermore, let $\mathcal{K}':\mathbb{R} \to [0,1]$ be a fuzzy number, and assume that 
	$ \langle ., . \rangle^\prime_{\alpha}:(0,1]\times X\times X \to \mathbb{C}$ is a fuzzy inner product corresponding to $\mathcal{K}'$ and satisfying the following formula:
	
	\begin{align}\label{def:fuzzy-inner-equivalence}		
		\begin{split}
			\mathcal{K}'\bigl(&\big|\langle x, y \rangle^\prime_\alpha \big|\bigr) \geq \alpha 
			\Leftrightarrow\\ 
			M_\alpha\, |\langle x,y\rangle^{\prime}|\leq \big|&\langle x, y \rangle^\prime_\alpha \big|
			\leq N_\alpha\, |\langle x,y\rangle^{\prime}|.
		\end{split}
	\end{align}
	Consequently, the following inequality holds:
	\begin{equation}\label{eq:fuzzy-inner-equivalence}
		\frac{A_{\alpha_1}}{N_{\alpha_2}}\, \big|\langle x, y \rangle^\prime_{\alpha_2} \big|
		\leq \big|\langle x, y \rangle_{\alpha_1} \big|
		\leq \frac{B_{\alpha_1}}{M_{\alpha_2}}\, \big|\langle x, y \rangle^\prime_{\alpha_2} \big|.
	\end{equation}
\end{proposition}

\begin{proof}
	By comparing formulas~\eqref{eq:fuzzyinerproduct-simple2} and~\eqref{def:fuzzy-inner-equivalence}, the inequality in~\eqref{eq:fuzzy-inner-equivalence} directly follows.
\end{proof}
\begin{corollary}
	Assume that $\langle ., . \rangle_{\alpha}:(0,1]\times X\times X \to \mathbb{C}$ is a fuzzy inner product satisfying formula~\eqref{eq:fuzzyinerproduct-simple2}. 
	Then, the following inequality holds:
	\begin{equation}
		\frac{A_{\alpha_1}}{B_{\alpha_2}} \, \big|\langle x, y \rangle_{\alpha_2} \big|
		\leq \big|\langle x, y \rangle_{\alpha_1} \big|
		\leq \frac{B_{\alpha_1}}{A_{\alpha_2}} \, \big|\langle x, y \rangle_{\alpha_2} \big|.
	\end{equation}
	
\end{corollary}
Similarly, the next result provides a concise description of how two fuzzy norms 
corresponding to different fuzzy numbers are related, by employing a common 
underlying norm to derive upper and lower bounds.

\begin{proposition}
	Suppose that $\norm{.}_\alpha:(0,1]\times X \to \mathbb{C}$ is a fuzzy norm satisfying formula~\eqref{eq:fuzzynorm-simple1}.  
	Furthermore, let $\mathcal{K}':\mathbb{R} \to [0,1]$ be a fuzzy number, and assume that 
	$\norm{.}^\prime_\alpha:(0,1]\times X \to \mathbb{C}$ is a fuzzy norm corresponding to $\mathcal{K}'$ and satisfying the following formula:
	\begin{align}
		\begin{split}
	\mathcal{K}'(\big|&\norm{x}^\prime_\alpha\big|\bigr) \geq \alpha 
			\Leftrightarrow\\ 
			M_\alpha\, &\norm{x}^{\prime}\leq \big|\norm{x}^\prime_\alpha\big| 
			\leq N_\alpha\, \norm{x}^{\prime}.
		\end{split}
	\end{align}
	Consequently, the following inequality holds:
	\begin{equation}
		\frac{C_{\alpha_1}}{N_{\alpha_2}}\, \big|\norm{x}^\prime_{\alpha_2}\big|
		\leq \big|\norm{x}_{\alpha_1}\big|
		\leq \frac{D_{\alpha_1}}{M_{\alpha_2}}\, \big|\norm{x}^\prime_{\alpha_2}\big|.
	\end{equation}
\end{proposition}

\begin{corollary}
	Assume that $Q:(0,1]\times X \to \mathbb{C}$ is a fuzzy norm satisfying formula~\eqref{eq:fuzzynorm-simple1}. 
	Then, the following inequality holds:
	\begin{equation}
		\frac{C_{\alpha_1}}{D_{\alpha_2}}\, \big|\norm{x}_{\alpha_2}\big|
		\leq \big|\norm{x}_{\alpha_1}\big|
		\leq \frac{D_{\alpha_1}}{C_{\alpha_2}}\, \big|\norm{x}_{\alpha_2}\big|.
	\end{equation}
	
\end{corollary}

\section{Conclusions and future research directions}
\label{sec:conclusions}
In this paper, we have made substantial progress toward resolving a long-standing open problem in fuzzy functional analysis: the apparent impossibility of constructing genuine, non-degenerate fuzzy inner products that preserve the essential geometric structure of classical Hilbert spaces. The crucial innovation is the introduction of \emph{ordered intervals}---a new class of  numbers equipped with fully cancellative addition and a natural weak partial order---thereby replacing the inherently non-cancellative standard interval and fuzzy arithmetics.
Using this ordered interval arithmetic, we proposed a new family of fuzzy inner products and induced fuzzy norms that are
\begin{itemize}
	\item  quasi-linear, with  explicit upper and lower bounds whenever exact equality is unavailable due to the absence of unrestricted cancellation,

	\item quasi-symmetric in an appropriate dual-bound sense, and
	\item geometrically rich: the Cauchy--Schwarz and Bessel inequalities hold in their full classical strength, while the parallelogram law and the polarization identity naturally appear as  inequalities.
\end{itemize}
From a conceptual viewpoint, the proposed formulation extends the classical notion of an inner product by associating to each ordered triple $(\alpha, x, y)$ a complex- (or real-) valued function that precisely encodes the interaction between the vectors $x$ and $y$ at the membership level $\alpha$. Similarly, the induced fuzzy norm assigns to each pair $(\alpha, x)$ a complex-valued function reflecting the graded magnitude of $x$ at level $\alpha$. Because these functions rigorously satisfy the membership constraints induced by the parameter $\alpha$, both the inner product and the norm admit a natural and transparent interpretation as \emph{graded, membership-level structures} that are fully consistent with the fuzzy environment—yet remain analytically tractable in a way that most previous fuzzy extensions are not.\\
A particularly noteworthy property of the proposed construction is the \emph{complete recovery of classical orthogonality}: two vectors $x$ and $y$ satisfy $\langle x, y \rangle = 0$  if and only if $\langle x, y \rangle_\alpha = 0$ for every $\alpha \in [0,1]$.
	\textbf{Most importantly}, the ordered interval-valued inner product introduced in this paper constitutes a \emph{genuine and faithful extension} of the classical inner product. It strictly includes the classical real- or complex-valued inner product as a degenerate subcase, generates truly non-degenerate fuzzy values, and fully recovers---or even reinforces via sharp upper and lower bounds---all fundamental algebraic and geometric properties of Hilbert spaces.\\
A striking illustration of this reinforcement is the following two-sided inequality, which holds for all vectors $x,y$ and scalar $k$ satisfying the natural ordering in the ordered interval structure:

	\[
\frac{|k| A_\alpha}{B_\alpha} \, \big|\langle y, x \rangle_\alpha \big|
\leq 
\big|\langle k x, y \rangle_\alpha \big|
\leq 
\frac{|k| B_\alpha}{A_\alpha} \, \big|\langle y, x \rangle_\alpha \big|.
\]
This sharp, explicit bound simultaneously generalises and strengthens the corresponding monotonicity and scaling properties of the classical inner product, which are usually taken for granted without explicit two-sided estimates. In sharp contrast to virtually all previous fuzzy inner product definitions---which achieved non-degeneracy only by sacrificing core Hilbert-space structure---the present construction is the first to deliver genuine fuzziness while preserving (and frequently enhancing) the richness of the classical theory.\\
Several natural and promising directions for future research arise from this work:
\begin{itemize}
	\item \textbf{Fuzzy orthogonality and orthonormal systems.} Capitalise on the exact recovery of classical orthogonality to construct orthonormal bases (in separable and non-separable settings), develop orthogonal projections, establish a Riesz representation theorem, and investigate frame theory with genuinely fuzzy frame bounds in ordered interval Hilbert spaces.
	
	\item \textbf{Spectral theory.} Study bounded, self-adjoint, normal, and compact operators on ordered interval Hilbert spaces and develop a comprehensive spectral theory.
	
	\item \textbf{Completeness and classification.} Determine completeness criteria for the induced fuzzy norm and fully characterise complete ordered interval Hilbert and Banach spaces.
	
	\item \textbf{Applications.} Explore applications in fuzzy quantum mechanics, fuzzy partial differential equations, similarity-based reasoning, and optimisation under severe (Knightian) uncertainty.
\end{itemize}
We believe that the ordered interval framework presented in this paper provides a solid, coherent, and geometrically faithful foundation for fuzzy functional analysis, finally overcoming limitations that have persisted for more than three decades.
Furthermore, ordered intervals themselves form a novel algebraic and ordered structure deserving independent investigation. Beyond the real ordered intervals considered here, natural generalisation include  \emph{complex ordered intervals} (with values in $\mathbb{C}$). Each variant exhibits distinctive arithmetic and order-theoretic properties with potential applications in discrete uncertainty modelling, complex-valued uncertain systems, and lattice-based approaches to uncertainty. A systematic study of this generalised structure is expected to yield new insights into ordered algebraic systems and their applications in mathematics and engineering.

\bibliographystyle{elsarticle-num}
\bibliography{references}
\end{document}